\documentclass[12pt,a4paper]{article}
    \usepackage{authblk}
    \usepackage{appendix}
    \usepackage{amsfonts}
    \usepackage{amsmath}
    \usepackage{amsthm}
    \usepackage{blkarray}
    \usepackage{stmaryrd}
    \usepackage{tikz-cd}
    \usepackage{hyperref}
    \usepackage{pict2e}
    \usepackage{adjustbox}
    \usepackage{xcolor}
    \usepackage{enumitem}

    \DeclareMathOperator{\rank}{rank}
    
    \theoremstyle{plain}%
    \newtheorem{theorem}{Theorem}
    \newtheorem{proposition}{Proposition}
    
    \theoremstyle{remark}%
    
    \newtheorem{remark}{Remark}%
    
    \theoremstyle{definition}%
    \newtheorem{definition}{Definition}%
    \newtheorem{lemma}{Lemma}%
    \newtheorem{example}{Example}%
    \newtheorem{corollary}{Corollary}%
    \makeatletter
    \def\th@plain{%
      \thm@notefont{}
      \itshape 
    }
    \def\th@definition{%
      \thm@notefont{}
      \normalfont 
    }
    \makeatother
    
\begin{document}
    
    \title{Classification of Lie algebras constructed from $\mathfrak{gl}_{m\vert n}$ via Derived Bracket}
    
    \author[1]{Luan Figueiredo\thanks{luanmat@ufmg.br}}

    \affil[1]{ICEX, Universidade Federal de Minas Gerais}
    \maketitle
    \begin{abstract}
Derived brackets provide a systematic mechanism for constructing algebraic structures from graded Lie superalgebras and play an important role in Poisson geometry, mathematical physics, and the theory of algebroids. In this paper, we study a family of Lie algebras obtained from the general linear Lie superalgebra $\mathfrak{gl}_{m\vert n}$ over a field $\mathbb{K}$ of characteristic zero through the derived bracket associated with an odd element $B \in \mathfrak{g}_1$ satisfying $B^2=0$.

We establish a complete structural and isomorphism classification of the resulting Lie algebras $\mathfrak{g}^B_{-1}$. For fixed dimensions $m$ and $n$, we prove that the isomorphism class of $\mathfrak{g}^B_{-1}$ is completely determined by the rank $r = \rank(B)$. More generally, for algebras arising from $\mathfrak{gl}_{m\vert n}$ and $\mathfrak{gl}_{p\vert q}$, we show that two such Lie algebras are isomorphic if and only if they have the same rank and satisfy $\{m,n\}=\{p,q\}$.

The classification is obtained through an explicit analysis of the Lie structure induced by the derived bracket and a complete computation of the Levi-Malcev decomposition. In particular, we determine the solvable radical and center and prove that the Levi factor is isomorphic to $\mathfrak{sl}(r)$.
\end{abstract}
    
    \noindent\textbf{Keywords:} Lie algebra; Lie superalgebra; derived bracket; Levi decomposition; classification.
    
    \noindent\textbf{MSC Classification:} 17B05, 17B70, 17B20, 17B30.
    
\section{Introduction}

Derived brackets provide a general framework for generating new algebraic structures from graded Lie algebras and Lie superalgebras. Although the terminology emerged later, constructions of this type appear in several areas of mathematics and mathematical physics, including Hamiltonian operators \cite{bib1}, BRST formalism \cite{bib7}, Poisson geometry \cite{bib4,bib5}, and the theory of Lie and Courant algebroids \cite{bib6,bib9}. Their systematic development has led to a unified perspective on how algebraic structures can be induced from graded or differential settings.

Classical examples illustrate the breadth of this construction. The Lie bracket of vector fields can be recovered from the Cartan identities through an iterated commutator expression, while Poisson brackets may be realized as derived brackets generated by bivector fields. These constructions demonstrate that derived brackets naturally encode nontrivial algebraic information inside larger graded structures.

In this work, we investigate a family of Lie algebras obtained from the general linear Lie superalgebra
\[
\mathfrak{g}=\mathfrak{gl}_{m\vert n}
\]
over a field $\mathbb{K}$ of characteristic zero. Endowing $\mathfrak{g}$ with its standard short $\mathbb{Z}$-grading \cite{bib3},
\[
\mathfrak{g}=\mathfrak{g}_{-1}\oplus\mathfrak{g}_{0}\oplus\mathfrak{g}_{1},
\]
and choosing an odd element $B\in\mathfrak{g}_{1}$ satisfying $B^2=0$, one obtains a Lie bracket on $\mathfrak{g}_{-1}$ through the derived construction
\[
\llbracket X,Y \rrbracket_B=[X,[B,Y]], \qquad X,Y\in\mathfrak{g}_{-1}.
\]
We denote the resulting Lie algebra by $\mathfrak{g}^B_{-1}$.

While the Lie property of this construction follows from general derived bracket theory \cite{bib5,bib8}, the structure and classification of the resulting algebras have not been previously described. The purpose of this paper is to provide a complete classification of these Lie algebras and to determine their internal structure explicitly.

Our first result concerns fixed dimensions $m$ and $n$. We prove that the Lie algebra $\mathfrak{g}^B_{-1}$ depends, up to isomorphism, only on the rank of the matrix representing $B$. More precisely, if $B_1$ and $B_2$ have the same rank, then the corresponding Lie algebras are isomorphic; otherwise they are not.

We then extend the classification across arbitrary dimensions. Given Lie algebras constructed respectively from $\mathfrak{gl}_{m\vert n}$ and $\mathfrak{gl}_{p\vert q}$, we prove that they are isomorphic precisely when the associated odd elements have the same rank and the unordered pairs of dimensions coincide:
\[
\mathfrak{g}^{B}_{-1} \simeq \mathfrak{h}^{H}_{-1} \iff \rank(B)=\rank(H) \text{ and } \{m,n\}=\{p,q\}.
\]

The key ingredient in the proof is an explicit computation of the Levi-Malcev decomposition of $\mathfrak{g}^B_{-1}$. We determine its solvable radical and center and show that its Levi factor is naturally isomorphic to $\mathfrak{sl}(r)$, where $r=\rank(B)$.

The paper is organized as follows. Section 2 recalls the necessary background on Lie superalgebras, Levi-Malcev theory, and derived brackets. Section 3 develops the structural analysis and establishes the classification for fixed dimensions. Section 4 extends the classification to arbitrary dimensions and presents illustrative examples.
      \section{Preliminaries}\label{sec2}

    \subsection{Lie Superalgebras}
    In this section we recall basic definitions of the theory of Lie algebras and Lie superalgebras. Throughout the paper we work over a field $\mathbb{K}$ of characteristic zero.

    \begin{definition}\label{d1}
    Let $G$ be an abelian group and $V = \bigoplus_{\alpha \in G} V_\alpha$ be a $G$-graded vector space. If $x \in V_\alpha$ then we say that $x$ is of homogeneous degree $\alpha$ and denote $\text{deg } x = \alpha$. From now on we denote $\mathbb{Z}_2 = \{\bar{0}, \bar{1}\}$.
    \end{definition}
    
    \begin{definition}\label{d2_new}
    A superalgebra $A = A_{\bar{0}} \oplus A_{\bar{1}}$ is a superspace over $\mathbb{K}$ such that for every $\overline{a}, \overline{b} \in \mathbb{Z}_2$:
    $$A_{\overline{a}} A_{{\overline{b}}} \subset A_{\overline{a+b}}.$$
    \end{definition}
    
    \begin{example}\label{e1_new}
    As usual we denote by $\text{End}\ V$ all linear transformations of vector space $V$. The vector space $\text{End}\ V$ is an associative algebra with respect to a composition of linear maps. In case that $V$ is a $G$-graded space, $\text{End}\ V$ admits the following $G$-grading:
    $$ \text{End}\ V = \bigoplus_{\alpha \in G} (\text{End}\ V)_\alpha, \quad (\text{End}\ V)_\alpha = \{f \in \text{End}\ V \mid f(V_s) \subseteq V_{s+\alpha}, \forall s \in G\}. $$
    In particular, for $G = \mathbb{Z}_2$ we obtain the following associative superalgebra:
    $$ \text{End}\ V = (\text{End}\ V)_{\bar{0}} \oplus (\text{End}\ V)_{\bar{1}}. $$
    We will consider only the case where $G = \mathbb{Z}_2$.
    \end{example}
    
    \begin{definition}\label{d2}
    A \textit{Lie superalgebra} $L$ is a superalgebra with a Lie superbracket operation $[\cdot, \cdot]$, satisfying:\\
    
    \noindent $i.)\ [a,b] = -(-1)^{(\text{deg } a)(\text{deg } b)}[b,a]$;\\
    $ii.)\ [a,[b,c]] = [[a,b],c] + (-1)^{(\text{deg } a)(\text{deg } b)}[b,[a,c]]$.
    \end{definition}
    
    \begin{example}
    \label{e1}
    We can define a Lie superalgebra structure on $\text{End}\ V$ in the following way:
    $$[X,Y] = X \circ Y - (-1)^{(\text{deg } X)(\text{deg } Y)} Y \circ X,$$
    for homogeneous $X,Y$. We denote this Lie superalgebra as $\mathfrak{gl}(V)$.
    \end{example}
    
    \begin{example}
    \label{e2}
    Let $V$ be a finite-dimensional superspace such that $\dim V_{\bar{0}} = m$ and $\dim V_{\bar{1}} = n$, and let us take operator $\Phi \in \mathfrak{gl}(V)$. Then we can choose basis $\beta_0 = \{e_1, \dots, e_m\}$ of $V_{\bar{0}}$ and $\beta_1 = \{e_{m+1}, \dots, e_{m+n}\}$ of $V_{\bar{1}}$. Then the pair $(\beta_0, \beta_1)$ is a basis for $V$. Now the operator $\Phi$ corresponds to the following matrix:
    \[ \Phi = \bordermatrix{ & m & n \cr
      m & A & B \cr
      n & C & D \cr} . \]
    We define the \textit{general linear Lie superalgebra} $\mathfrak{gl}_{m\vert n}$ as the $(m+n) \times (m+n)$ matrices over $\mathbb{K}$ together with the $\mathbb{Z}_2$-grading above. More precisely,
    $$ \mathfrak{gl}_{m\vert n} = \left\{ \begin{pmatrix} A & B \\ C & D \end{pmatrix} \right\},\ (\mathfrak{gl}_{m\vert n})_{\bar{0}} = \left\{ \begin{pmatrix} A & 0 \\ 0 & D \end{pmatrix} \right\},\ (\mathfrak{gl}_{m\vert n})_{\bar{1}} = \left\{ \begin{pmatrix} 0 & B \\ C & 0 \end{pmatrix} \right\}. $$
    Sometimes we denote $\mathfrak{gl}_{m\vert n}$ simply by $\mathfrak{g}$. Observe that if the same vector space $V = V_{\bar{0}} \oplus V_{\bar{1}}$ is considered as $\mathbb{Z}$-graded, this is $V_0 = V_{\bar{0}}$, $V_{-1} \oplus V_1 = V_{\bar{1}}$ and $V_k = \{0\}, \forall k \notin \{-1, 0, 1\}$, then $\mathfrak{gl}_{m\vert n}$ possesses the following $\mathbb{Z}$-grading:
    $$ \begin{pmatrix} A & 0 \\ 0 & D \end{pmatrix} \in \mathfrak{g}_0, \quad \begin{pmatrix} 0 & 0 \\ C & 0 \end{pmatrix} \in \mathfrak{g}_{-1}, \quad \begin{pmatrix} 0 & B \\ 0 & 0 \end{pmatrix} \in \mathfrak{g}_1. $$
    In particular, we see that $[\mathfrak{g}_{-1}, \mathfrak{g}_{-1}] \subseteq \mathfrak{g}_{-2} = \{0\}$ and $[\mathfrak{g}_1, \mathfrak{g}_1] \subseteq \mathfrak{g}_2 = \{0\}$.
    \end{example}
    \subsection{The Levi-Malcev Decomposition}
Let us recall the Levi-Malcev theorem. For that, we need some definitions from Lie algebra theory.
    
    \begin{definition}
    \label{d3}
    The radical of a finite-dimensional Lie algebra $L$, denoted $\mathfrak{r}(L)$, is its maximal solvable ideal. A Lie algebra $L$ is called semisimple if its radical $\mathfrak{r}$ is equal to zero.
    \end{definition}
    For ease of reading, we omit the $(L)$ in $\mathfrak{r}(L)$.
    
    \begin{lemma}
    \label{l1}
    Let $L$ be a finite dimensional Lie algebra over $\mathbb{K}$ of characteristic zero. Let $\mathfrak{r}$ be its radical and let us consider a solvable ideal $\mathfrak{a}$. Then\\
    a) $L/\mathfrak{r}$ is semisimple;\\
    b) A quotient $L/\mathfrak{a}$ is semisimple if and only if $\mathfrak{r} = \mathfrak{a}$.
    \end{lemma}
    
    \begin{proof}
    a) Let $\pi: L \to L/\mathfrak{r}$ be the canonical map and $J$ be a solvable ideal of $L/\mathfrak{r}$. Then $U = \pi^{-1}(J)$, $\mathfrak{r} \subseteq U$ and $U/\mathfrak{r} \simeq J$. As $U$ is solvable and $\mathfrak{r}$ is maximal, $U/\mathfrak{r} = 0$, therefore, $J = 0$;\\
    b) As $L/\mathfrak{a}$ is semisimple, $\mathfrak{r}/\mathfrak{a} = 0$ then $\mathfrak{r} = \mathfrak{a}$. Conversely, if $\mathfrak{r} = \mathfrak{a}$ then, by the item ``a", $L/\mathfrak{a}$ is semisimple.
    \end{proof}
    
    The Levi-Malcev Theorem states that a finite dimensional Lie algebra $L$ is a semidirect sum of a semisimple subalgebra and the solvable radical of $L$. The semisimple subalgebra, also called Levi subalgebra, is unique up to an automorphism.

    \begin{theorem}
    \label{t1}
    (Levi-Malcev) Let $L$ be a finite dimensional Lie algebra. If $L$ is not solvable, then there exists a semisimple subalgebra $\mathfrak{s}$ of $L$ such that $L = \mathfrak{s} \oplus \mathfrak{r}$ and $\mathfrak{s} \simeq L/\mathfrak{r}$.
    Furthermore, if $\mathfrak{s}$ and $\mathfrak{s}'$ are semisimple subalgebras of $L$ with $L = \mathfrak{s} \oplus \mathfrak{r} = \mathfrak{s}' \oplus \mathfrak{r}$, then there exists an automorphism $\sigma$ of $L$ such that $\sigma(\mathfrak{s}) = \mathfrak{s}'$.
    \end{theorem}
    \begin{proof}
    See \cite{bib2} for details.
    \end{proof}
    
    \subsection{Derived Brackets}
    Let $L = (V,[\cdot ,\cdot ])$ be a graded Lie superalgebra with bracket of degree $n$, {i. e.}, $[V_\alpha, V_\beta] \subset V_{(\alpha + \beta + n)}$ and $D$ be a  derivation of $L$.
    The derived bracket by $D$ can be defined on $L$  as: $$(x,y) \in V \times V \mapsto (-1)^{(n+ deg\ x)+1}[Dx,y] \in V.$$
    
    For our purposes we will use a specific version of this definition, which is valid when $D$  is an interior derivation of $L$ by an odd element $B$ of square zero and $x,y$ have odd degree. {In this case, the expression above becomes }
    $$(x,y) \in V \times V \mapsto [[B,x],y] \in V.$$
    In this paper we will consider Lie superalgebras with bracket of degree $0$, and we rewrite the definition equivalently as
    $$0=[B,[x,y]] = [[B,x],y] - [x,[B,y]] \implies [[B,x],y] = [x,[B,y]]. $$ 
    For more details, see \cite{bib5}.
    
    \begin{definition}
    \label{d4}
    Let $L=L_{-1}\oplus L_{0} \oplus L_{1}$ be a $\mathbb{Z}$-graded Lie superalgebra over $\mathbb{K}$ with bracket $[\cdot,\cdot]$ and $B\in L_{1}$. We can define the following bilinear map $\llbracket \cdot ,\cdot \rrbracket_{B}: L_{-1} \times L_{-1} \to$ {$L_{-1}$} \[ \llbracket x,y\rrbracket_{B} = [x,[B,y]],\] We call this new bracket the derived bracket of $[\cdot  ,\cdot ]$ by $B$.
    \end{definition}

    Sometimes for simplicity we will denote $\llbracket ., .\rrbracket_{B}$ simply by $\llbracket ., .\rrbracket$.
The following theorem can be deduced from the results of \cite{bib5, bib8}.
    \begin{theorem}
    \label{t3}
    Let $L=L_{-1}\oplus L_0\oplus L_1$ be a $\mathbb{Z}$-graded Lie superalgebra and let $\llbracket\cdot,\cdot\rrbracket_B$ be the derived bracket by $B$. Then $(L_{-1},\llbracket\cdot,\cdot\rrbracket_B)$ is a Lie algebra.
    \end{theorem}

    \begin{proof}
    We verify directly that the bracket
    \[
    \llbracket X,Y\rrbracket_B=[X,[B,Y]]
    \]
    is skew-symmetric and satisfies the Jacobi identity for all $X,Y,Z\in L_{-1}$.
    
    Since $L_{-1}$ is abelian, we have $[X,Y]=0$ for all $X,Y\in L_{-1}$. Hence, by the graded Jacobi identity,
    \[
    \llbracket X,Y\rrbracket_B=[X,[B,Y]]
    = [[X,B],Y]-[B,[X,Y]]
    = [[X,B],Y].
    \]
    Using the skew-symmetry of the Lie superbracket and the fact that $X,B$ are homogeneous, we obtain
    \[
    [[X,B],Y]=-[Y,[X,B]]=-[Y,[B,X]]=-\llbracket Y,X\rrbracket_B.
    \]
    Therefore, $\llbracket \cdot,\cdot\rrbracket_B$ is anticommutative.
    
    Now let $X,Y,Z\in L_{-1}$. We compute
    \[
    \llbracket X,\llbracket Y,Z\rrbracket_B\rrbracket_B=[X,[B,[Y,[B,Z]]]].
    \]
    Since $[B,B]=0$, the graded Jacobi identity gives
    \[
    [B,[Y,[B,Z]]] = [[B,Y],[B,Z]].
    \]
    Thus
    \[
    \llbracket X,\llbracket Y,Z\rrbracket_B\rrbracket_B=[X,[[B,Y],[B,Z]]].
    \]
    Applying the Jacobi identity once more, we get
    \[
    [X,[[B,Y],[B,Z]]]
    =
    [[X,[B,Y]],[B,Z]]+[[B,Y],[X,[B,Z]]].
    \]
    The first term is
    \[
    [[X,[B,Y]],[B,Z]] = [\llbracket X,Y\rrbracket_B,[B,Z]] = \llbracket \llbracket X,Y\rrbracket_B,Z\rrbracket_B.
    \]
    For the second term, using the already proved anticommutativity, we may rewrite it as
    \[
    [[B,Y],[X,[B,Z]]]
    =
    -[[X,[B,Z]],[B,Y]]
    =
    \llbracket Y,\llbracket X,Z\rrbracket_B\rrbracket_B.
    \]
    Hence
    \[
    \llbracket X,\llbracket Y,Z\rrbracket_B\rrbracket_B = \llbracket \llbracket X,Y\rrbracket_B,Z\rrbracket_B + \llbracket Y,\llbracket X,Z\rrbracket_B\rrbracket_B,
    \]
    which is the Jacobi identity for the bracket $\llbracket\cdot,\cdot\rrbracket_B$.
    
    Therefore $(L_{-1},\llbracket\cdot,\cdot\rrbracket_B)$ is a Lie algebra.
    \end{proof}
    
    \begin{remark}
    Our main interest is when $L = \mathfrak{gl}_{m\vert n}.$ From now on, we denote the Lie algebra  $(\mathfrak{g}_{-1},\llbracket \cdot,\cdot \rrbracket_B)$ as $\mathfrak{g}^{B}_{-1}$.
    \end{remark}
    \section{Structure and Levi-Malcev Decomposition of $\mathfrak{g}_{-1}^{B}$}\label{sec3}
    \subsection{Reduction to Normal Form}\label{sec3.1}
    Our results are the classification of Lie algebras $\mathfrak{g}^{B}_{-1}$ where $\mathfrak{g}_{-1}$ is of size $m\times n$ and $B\in \mathfrak{g}_1$ is of rank $r$, as in Lemma~\ref{l2} below. First, we consider the case when the values of $m,n$ and $r$ are fixed, and later when they are arbitrary.
    
    \paragraph{Block notation.}
Matrices in $\mathfrak{gl}_{m\vert n}$ are written in block form
\[
X=
\bordermatrix{
&m&n\cr
m&\mathbf a&\mathbf b\cr
n&\mathbf c&\mathbf d
}.
\]

After fixing a rank-$r$ normal form for $B$, matrices in
$\mathfrak g_{-1}$ are further decomposed according to
\[
(r,n-r)\times(r,m-r):
\]

\[
\mathbf x=
\bordermatrix{
&r&m-r\cr
r&x_{11}&x_{12}\cr
n-r&x_{21}&x_{22}
}.
\]

This block decomposition will be used throughout the structural and
classification arguments.

The following lemma is a standard result from linear algebra.
\begin{lemma}\label{l2}
Let $B$ be a block matrix of the form
\[ B = \bordermatrix{ & m & n \cr m & 0 & \mathbf{b} \cr n & 0 & 0 \cr}, \]
where $\mathbf{b}$ is a matrix of size $m \times n$ and $\rank \mathbf{b} = r$.
Then there exist invertible matrices $P_m$ and $P_n$ such that
$P_m \mathbf{b} P_n = \mathbf{b}'$, where
$\mathbf{b}' = \left(\begin{smallmatrix} I_r & 0 \\ 0 & 0 \end{smallmatrix}\right)$.
We call $B'$ the \textit{normal form of $B$}.
\end{lemma}
\begin{proof}
Since $\rank \mathbf{b} = r$, elementary row operations reduce $\mathbf{b}$
to a matrix with $r$ pivot columns. These operations correspond to left
multiplication by an invertible matrix $P_m$. Applying elementary column
operations to the result, corresponding to right multiplication by an
invertible $P_n$, places the pivots along the leading diagonal and
eliminates all other entries, yielding $\mathbf{b}' = \bigl(\begin{smallmatrix}
I_r & 0 \\ 0 & 0 \end{smallmatrix}\bigr)$.
\end{proof}

    For convenience, from now on we will use only the normal form of $B$ and we say that $B'$ has rank $r$.
    In the following lemma we describe the Lie algebra structure of $\mathfrak{g}^{B'}_{-1}$ explicitly.
    \begin{lemma}
    \label{l3}
Let $B \in \mathfrak{g}_{1}$, $X,Y \in \mathfrak{g}_{-1}$ and let $\rank B=r$. Then \\
    a) For any $B$ we have that \[ \llbracket X , Y \rrbracket_B =  [X,[B,Y]] = \bordermatrix{ & m & n \cr
      m & 0 & 0 \cr
      n & \textbf{xby} - \textbf{ybx} & 0 \cr}. \]
    
    b) Let $\textbf{x}, \textbf{y}$ and $\textbf{b}$ be in the following form:
    \[ \textbf{x} = \begin{blockarray}{ccc}
      & r & m-r \\
      \begin{block}{c(cc)}
        r & x_{11} & x_{12} \\
        n-r & x_{21} & x_{22} \\
      \end{block}
    \end{blockarray},\quad \textbf{y} = \begin{blockarray}{ccc}
      & r & m-r \\
      \begin{block}{c(cc)}
        r & y_{11} & y_{12} \\
        n-r & y_{21} & y_{22} \\
      \end{block}
    \end{blockarray},\quad \textbf{b} = \begin{blockarray}{ccc}
      & r & n-r \\
      \begin{block}{c(cc)}
        r & I_r & 0 \\
        m-r & 0 & 0 \\
      \end{block}
    \end{blockarray}. \] Then \[
     \textbf{xby} - \textbf{ybx} = \begin{blockarray}{ccc}
      & r & m-r \\
      \begin{block}{c(cc)}
        r & x_{11}y_{11}-y_{11}x_{11} & x_{11}y_{12}-y_{11}x_{12} \\
        n-r & x_{21}y_{11}-y_{21}x_{11} & x_{21}y_{12}-y_{21}x_{12} \\
      \end{block}
    \end{blockarray}. \]
    \end{lemma}
    \begin{proof}
    Calculating the bracket, we have:
    \begin{align*}
    [B,Y] &= \begin{pmatrix} 0 & \textbf{b} \\ 0 & 0 \end{pmatrix} \begin{pmatrix} 0 & 0 \\ \textbf{y} & 0 \end{pmatrix} - (-1)^{(1)(-1)} \begin{pmatrix} 0 & 0 \\ \textbf{y} & 0 \end{pmatrix} \begin{pmatrix} 0 & \textbf{b} \\ 0 & 0 \end{pmatrix} \\
    &= \begin{pmatrix} \textbf{by} & 0 \\ 0 & \textbf{yb} \end{pmatrix}.
    \end{align*}
    Applying the bracket with $X$, we have:
    \begin{align*}
    [X,[B,Y]] &= \begin{pmatrix} 0 & 0 \\ \textbf{x} & 0 \end{pmatrix} \begin{pmatrix} \textbf{by} & 0 \\ 0 & \textbf{yb} \end{pmatrix} - \begin{pmatrix} \textbf{by} & 0 \\ 0 & \textbf{yb} \end{pmatrix} \begin{pmatrix} 0 & 0 \\ \textbf{x} & 0 \end{pmatrix} \\
    &= \begin{pmatrix} 0 & 0 \\ \textbf{xby} & 0 \end{pmatrix} - \begin{pmatrix} 0 & 0 \\ \textbf{ybx} & 0 \end{pmatrix} = \begin{pmatrix} 0 & 0 \\ \textbf{xby}-\textbf{ybx} & 0 \end{pmatrix}.
    \end{align*}

                              The proof for item ``b'' follows from the matrix multiplication of $(\textbf{xby} - \textbf{ybx})$.

    \end{proof}
    The following theorem states that any Lie algebras  $\mathfrak{g}^B_{-1}$ is  isomorphic to $\mathfrak{g}^{B'}_{-1}$, where $B'$ is as in Lemma~\ref{l2}.
    \begin{theorem}
    \label{t4}
 Let $\mathfrak{g}^{B}_{-1}$ be a Lie algebra, $B\in \mathfrak{g}_{1}$, $\rank B = \rank B' = r$ and $B'$ be as in Lemma~\ref{l2}. Then there is an isomorphism of Lie algebras:
    $${\mathfrak{g}^{B}_{-1}} \xrightarrow{\sim} {\mathfrak{g}^{B'}_{-1}}.$$
    \end{theorem}
    \begin{proof} Consider the matrix $P = \left( \begin{array}{cc}
                                 P_m & 0 \\
                                 0 & {P_n}^{-1} \\
                                \end{array}
                              \right)$, where $P_m$ and $P_n$ are from Lemma~\ref{l2}. \\
    Let $\phi_P : \mathfrak{gl}_{m\vert n} \to \mathfrak{gl}_{m\vert n}, X \mapsto P X P^{-1}$ and
    $\pi: \mathfrak{gl}_{m\vert n} \to \mathfrak{g}_{-1}$ be the natural projection.
    We prove the theorem by constructing a commutative diagram of $\mathfrak{gl}_{m\vert n}$ and $\mathfrak{g}_{-1}$, using the mappings defined above.
To achieve that we need to prove that each $\mathfrak{g}_{-1}$, $\mathfrak{g}_0$ and $\mathfrak{g}_1$ are invariant by $\phi_P$.
    Let \[ T = \bordermatrix{ & m & n \cr
      m & \textbf{x} & \textbf{y} \cr
      n & \textbf{z} & \textbf{w} \cr}, \quad T \in \mathfrak{gl}_{m\vert n}. \] Then
      $$\phi_P (T) =
    \left( \begin{array}{cc}
                                 P_m & 0 \\
                                 0 & P_n^{-1} \\
                                \end{array}
                              \right) \left( \begin{array}{cc}
                                 \textbf{x} & \textbf{y} \\
                                 \textbf{z} & \textbf{w} \\
                                \end{array}
                              \right) \left( \begin{array}{cc}
                                 P_m^{-1} & 0 \\
                                 0 &  {P_n}\\
                                \end{array}
                              \right) = \left( \begin{array}{cc}
                                 {P_m}\textbf{x}P_m^{-1} & P_m\textbf{y}P_n \\
                                 P_n^{-1}\textbf{z}P_m^{-1} & P_n^{-1}\textbf{w}P_n \\
                                \end{array}
                              \right).$$\\
    Now we see that $\phi_P (\mathfrak{g}_{0}) \subseteq \mathfrak{g}_{0}$, $\phi_P (\mathfrak{g}_{1}) \subseteq \mathfrak{g}_{1}$ and $\phi_P (\mathfrak{g}_{-1}) \subseteq \mathfrak{g}_{-1}$. Therefore the following diagram is commutative:
    \[
    \begin{tikzcd}[column sep=large, row sep=large]
      \mathfrak{gl}_{m\vert n} \arrow[r, "\phi_P"] \arrow[d, "\pi"'] & \mathfrak{gl}_{m\vert n} \arrow[d, "\pi"] \\
      \mathfrak{g}_{-1} \arrow[r, "\phi_P"] & \mathfrak{g}_{-1}
    \end{tikzcd}
    \]

    It remains to prove that $\phi_P$ is a homomorphism of Lie algebras, i.e., $\phi_P (\llbracket X, Y \rrbracket_{B}) = \llbracket \phi_P (X), \phi_P (Y) \rrbracket_{B'}$.\\
    Let $X, Y \in \mathfrak{g}_{-1}$.
    By Lemma~\ref{l3}, we have
    \begin{align*}
      \phi_P(\llbracket X,Y \rrbracket _{B}) &= \begin{pmatrix} P_m & 0 \\ 0 & P_n^{-1} \end{pmatrix} \begin{pmatrix} 0 & 0 \\ \textbf{xby}-\textbf{ybx} & 0 \end{pmatrix} \begin{pmatrix} P_m^{-1} & 0 \\ 0 & P_n \end{pmatrix} \\
      &= \begin{pmatrix} 0 & 0 \\ P_{n}^{-1}(\textbf{xby} - \textbf{ybx})P_m^{-1} & 0 \end{pmatrix}.
    \end{align*}
    Writing $\textbf{b}$ as $P_{m}^{-1} \textbf{b}' P_{n}^{-1} $ we get:
    \begin{align*}
      &= \begin{pmatrix} 0 & 0 \\ (P_{n}^{-1} \textbf{x}P_m^{-1}) \textbf{b}' (P_{n}^{-1} \textbf{y} P_m^{-1}) - (P_{n}^{-1} \textbf{y} P_m^{-1}) \textbf{b}' (P_n^{-1} \textbf{x} P_m^{-1}) & 0 \end{pmatrix} \\
      &= \llbracket P_{n}^{-1}X P_m^{-1}, P_{n}^{-1} YP_m^{-1} \rrbracket _{B'} = \llbracket \phi_P (X),\phi_P (Y) \rrbracket _{B'}.
\end{align*} Therefore $\phi_P$ is an isomorphism of Lie algebras.
                                       \end{proof}

    \subsection{The Levi-Malcev Components}
    Below we obtain the center, radical and Levi subalgebra of $\mathfrak{g}^{B'}_{-1}$.
    
    \begin{remark}
      By Lemma~\ref{l2}, every matrix $B$ is equivalent to a normal form $B'$. By Theorem~\ref{t4},
the Lie algebras associated to $B$ and $B'$ are isomorphic. Therefore, replacing
$B$ by $B'$, we may assume that $B$ is in normal form. For simplicity, we continue
to denote this normal form by $B$.
    \end{remark}

   \begin{corollary}\label{c1}
Let $\mathfrak{g}_{-1}$ be of size $m \times n$, and let $B$ be of rank $r$. Then:
\begin{enumerate}[label=\alph*)]

\item If $r = m = n$, the center consists of the scalar matrices
\[
Z(\mathfrak{g}^B_{-1}) = \{\,\alpha I_r : \alpha \in \mathbb{K}\,\}.
\]

\item If $r \neq m$ and $r \neq n$, then
\[
Z(\mathfrak{g}^B_{-1})
=
\left\{
\begin{pmatrix} 0 & 0 \\ 0 & y_{22} \end{pmatrix}
\;:\; y_{22} \in M_{(n-r)\times(m-r)}(\mathbb{K})
\right\}.
\]

\item If $r = m$ or $r = n$ (but not both), then $Z(\mathfrak{g}^B_{-1}) = \{0\}$.

\end{enumerate}
\end{corollary}

\begin{proof}
We determine the center by solving $\llbracket X, W \rrbracket_B = 0$ for all
$X \in \mathfrak{g}^B_{-1}$. Writing $X$ and $W$ in the block decomposition
of Lemma~\ref{l3} and applying Lemma~\ref{l2}, this condition becomes
\begin{equation}\label{eq:center}
\begin{pmatrix}
x_{11}w_{11} - w_{11}x_{11} & x_{11}w_{12} - w_{11}x_{12} \\
x_{21}w_{11} - w_{21}x_{11} & x_{21}w_{12} - w_{21}x_{12}
\end{pmatrix}
= 0,
\end{equation}
which must hold for all choices of blocks $x_{11}, x_{12}, x_{21}$.

\medskip
\textbf{Case (a):} If $r = m = n$, then $W = w_{11}$ is the only block and
\eqref{eq:center} reduces to $x_{11}w_{11} = w_{11}x_{11}$ for all
$x_{11} \in M_r(\mathbb{K})$. By Schur's lemma, $w_{11} = \alpha I_r$ for some
$\alpha \in \mathbb{K}$.

\medskip
\textbf{Case (b):} If $r \neq m$ and $r \neq n$, all four blocks are present.
The $(1,1)$-entry of \eqref{eq:center} gives $x_{11}w_{11} = w_{11}x_{11}$ for
all $x_{11}$, so $w_{11} = \alpha I_r$. Substituting into the $(1,2)$-entry yields
$x_{11}w_{12} = \alpha x_{12}$ for all $x_{11}$ and $x_{12}$. Taking $x_{11} = 0$
forces $\alpha = 0$, and hence $w_{11} = 0$; the equation then reduces to
$x_{11}w_{12} = 0$ for all $x_{11}$, giving $w_{12} = 0$. The $(2,1)$-entry
similarly gives $w_{21} = 0$. Therefore
\[
W = \begin{pmatrix} 0 & 0 \\ 0 & w_{22} \end{pmatrix},
\]
with $w_{22}$ free, which is precisely the set described in item (b).

\medskip
\textbf{Case (c):} Suppose $r = m$ (the case $r = n$ is symmetric). The
right block column of \eqref{eq:center} (corresponding to $m-r = 0$) does not exist, so the system reduces to the two conditions
\[
x_{11}w_{11} = w_{11}x_{11} \quad \text{and} \quad x_{21}w_{11} = w_{21}x_{11}.
\]
The first condition gives $w_{11} = \alpha I_r$ for some $\alpha \in \mathbb{K}$. Substituting into the second yields $\alpha x_{21} = w_{21}x_{11}$. Taking $x_{11} = 0$ forces $\alpha x_{21} = 0$ for all $x_{21}$, which implies $\alpha = 0$ and thus $w_{11} = 0$. The condition then becomes $w_{21}x_{11} = 0$ for all $x_{11}$, which gives $w_{21} = 0$.
If instead $r = n$, an analogous argument yields $w_{11} = 0$ and $w_{12} = 0$.
In either sub-case $W = 0$, so $Z(\mathfrak{g}^B_{-1}) = \{0\}$.
\end{proof}

    \begin{corollary}
    \label{c2}
     Consider the Lie algebra ${\mathfrak{g}^{B}_{-1}}$ and let rank $B = r$. 
     Let \[ R = \left\{ \bordermatrix{ & r & m-r \cr
       r & \alpha I_r & X_{12} \cr
      n-r & X_{21} & X_{22} \cr} \mathrel{\Bigg\vert} \alpha \in \mathbb{K} \right\} \]
    and
    \[ S = \left\{ \bordermatrix{ & r & m-r \cr
       r & X_{11} & 0 \cr
      n-r & 0 & 0 \cr} \mathrel{\Bigg\vert} X_{11} \text{ is an } r \times r \text{ matrix with } \operatorname{tr}(X_{11}) = 0 \right\}. \]
    Then\\
    a) $R$ is solvable;\\
    b) $S$ is semisimple;\\
    c) $R = \mathfrak{r}(\mathfrak{g}^{B}_{-1})$ and $S = \mathfrak{s}(\mathfrak{g}^{B}_{-1})$, forming the Levi-Malcev decomposition $\mathfrak{g}^{B}_{-1} = S \oplus R$.
    \end{corollary}
    \begin{proof} 
    Let $R^{(1)} = \llbracket R, R \rrbracket_{B}$, $R^{(2)} = \llbracket R^{(1)}, R^{(1)} \rrbracket_{B}$, and similarly define $R^{(n)}$ for $n \in \mathbb{N}$ as the derived series of $R$ with respect to the bracket $\llbracket \cdot, \cdot \rrbracket_{B}$.\\
    $a$) Let $X, Y \in R$ with blocks $\alpha_1 I_r, X_{12}, X_{21}, X_{22}$ and $\alpha_2 I_r, Y_{12}, Y_{21}, Y_{22}$, respectively. Then, by Lemma~\ref{l3} we have
    \begin{align*}
    \llbracket X, Y \rrbracket_{B} &= \bordermatrix{ & r & m-r \cr
     r & (\alpha_1 \alpha_2 -\alpha_2 \alpha_1)I_r & \alpha_1 Y_{12}-\alpha_2 X_{12} \cr
     n-r & \alpha_2 X_{21} - \alpha_1 Y_{21} & X_{21} Y_{12}-Y_{21} X_{12} \cr} \\
     &= \bordermatrix{ & r & m-r \cr
     r & 0 & \alpha_1 Y_{12}-\alpha_2 X_{12} \cr
     n-r & \alpha_2 X_{21} - \alpha_1 Y_{21} & X_{21} Y_{12}-Y_{21} X_{12} \cr}.
    \end{align*}
    Let $X', Y' \in R^{(1)}$. Since their upper-left $r \times r$ blocks are zero ($\alpha'_1 = \alpha'_2 = 0$), we have
    \[ \llbracket X', Y' \rrbracket_{B} = \bordermatrix{ & r & m-r \cr
     r & 0 & 0 \cr
     n-r & 0 & X'_{21} Y'_{12}-Y'_{21} X'_{12} \cr}. \]
               
    Observe that $R^{(2)} \subseteq Z(\mathfrak{g}^{B}_{-1})$, therefore $R^{(3)} = \{0\}$ and $R$ is solvable.\\
    $b$) Observe that $S \simeq \mathfrak{sl}(r)$, therefore $S$ is semisimple, and simple for $r \ge 2$. (Note that when $r=1$, $\mathfrak{sl}(1) = \{0\}$ is the trivial algebra, and the Levi-Malcev decomposition becomes degenerate).\\
    $c$) As a vector space, $\mathfrak{g}^{B}_{-1} = S \oplus R$ with $S \cap R = \{0\}$. Since $S \simeq \mathfrak{g}^{B}_{-1}/R$ is semisimple and $R$ is a solvable ideal, it follows from Lemma~\ref{l1} b) that $R = \mathfrak{r}(\mathfrak{g}^{B}_{-1})$. Consequently, $S = \mathfrak{s}(\mathfrak{g}^{B}_{-1})$.
    \end{proof}
    \subsection{Isomorphism Classification for Fixed Dimensions}
    The following theorem is our first result: we can classify up to isomorphism all Lie algebras ${\mathfrak{g}^{B}_{-1}}$ using the possible integer values $r_i$ for rank $B$.

    \begin{theorem}
    \label{t5}
    Assume $(m, n) \neq (1, 1)$. Consider the Lie algebras ${\mathfrak{g}^{B_i}_{-1}}$ of size $m \times n$ and let $r_i, r_j \in \{0,\dots,\min\{m,n\}\}$ be the ranks of $B_i$ and $B_j$ respectively. Then we have
    $$\mathfrak{g}^{B_i}_{-1} \simeq \mathfrak{g}^{B_j}_{-1} \Leftrightarrow r_i = r_j.$$
    \end{theorem}

    \begin{remark}
    When $m = n = 1$, both possible ranks $r = 0$ and $r = 1$ yield the unique one-dimensional abelian Lie algebra, so the classification is trivial in this case.
    \end{remark}

    \begin{proof} Consider ${\mathfrak{g}^{B_1}_{-1}}$ and ${\mathfrak{g}^{B_2}_{-1}}$, where $\rank B_1 = r_1$, $\rank B_2 = r_2$ and $r_1 \ne r_2$. Let us prove that these algebras are not isomorphic. 
    
    First, consider the case where one of the ranks is zero, say $r_1 = 0$. Then $B_1 = 0$ and the derived bracket is identically zero, making $\mathfrak{g}^{B_1}_{-1}$ abelian. Since $(m,n)\neq(1,1)$, we have $\dim\mathfrak{g}_{-1} = mn > 1$, so for $r_2 > 0$ the algebra $\mathfrak{g}^{B_2}_{-1}$ is non-abelian (by Lemma~\ref{l3}, one can find $X, Y \in \mathfrak{g}_{-1}$ with $\llbracket X,Y\rrbracket_{B_2} \neq 0$) and hence not isomorphic to $\mathfrak{g}^{B_1}_{-1}$.
    
    Thus, we may assume $r_1, r_2 \ge 1$. Theorem~\ref{t1} implies that there are the following Levi decompositions 
\[{\mathfrak{g}^{B_1}_{-1}} = \mathfrak{s}(\mathfrak{g}^{B_1}_{-1}) \oplus \mathfrak{r}(\mathfrak{g}^{B_1}_{-1}), \quad \mathfrak{g}^{B_2}_{-1} = \mathfrak{s}(\mathfrak{g}^{B_2}_{-1}) \oplus \mathfrak{r}(\mathfrak{g}^{B_2}_{-1}). \]
    From the fact that $\mathfrak{s}({\mathfrak{g}^{B_i}_{-1}}) \simeq \mathfrak{sl}(r_i)$, $\dim \mathfrak{sl}(r_i)$ = $r_i^2 - 1$ and $ r_1 \ne r_2$, we conclude that $\dim \mathfrak{s}(\mathfrak{g}^{B_1}_{-1}) \ne \dim \mathfrak{s}(\mathfrak{g}^{B_2}_{-1})$, therefore $\mathfrak{s}(\mathfrak{g}^{B_1}_{-1}) \not\simeq \mathfrak{s}(\mathfrak{g}^{B_2}_{-1})$. From Theorem~\ref{t1}, we know that the Levi subalgebra is unique up to an automorphism, hence the Lie algebras $\mathfrak{g}^{B_1}_{-1}$ and $\mathfrak{g}^{B_2}_{-1}$ cannot be isomorphic. This fact together with Theorem~\ref{t4} implies that $\mathfrak{g}^{B_1}_{-1} \simeq \mathfrak{g}^{B_2}_{-1}$ if and only if $r_1 = r_2$. Therefore, each Lie algebra ${\mathfrak{g}^{B_{i}}_{-1}}$ is uniquely associated with each parameter $r_i$.
    \end{proof}
    \section{Isomorphism Classification for Arbitrary Dimensions}\label{sec4}

    Now we consider the classification across arbitrary dimensions. Let $\mathfrak{gl}_{p\vert q} = \mathfrak{h}_{-1} \oplus \mathfrak{h}_{0} \oplus \mathfrak{h}_{1}$, and let $H \in \mathfrak{h}_{1}$ be of rank $r_2$, which we assume is in normal form just as we did for $B$. As before, denote the Lie algebra $(\mathfrak{h}_{-1}, \llbracket\cdot, \cdot\rrbracket_{H})$ as $\mathfrak{h}^{H}_{-1}$.
    
    In the remainder of this section, we establish the conditions under which the Lie algebras $\mathfrak{g}^{B}_{-1}$ and $\mathfrak{h}^{H}_{-1}$ are isomorphic. Specifically, letting $r_1$ be the rank of $B$ and $r_2$ be the rank of $H$, we will show that these algebras are isomorphic if and only if they share the same rank and their dimension parameters agree up to permutation:
    \[
    \mathfrak{g}^{B}_{-1} \simeq \mathfrak{h}^{H}_{-1} \iff r_1 = r_2 \text{ and } \{m,n\} = \{p,q\}.
    \]
    
    \subsection{The Supertranspose Isomorphism}
    In order to achieve this generalization we construct an isomorphism between $\mathfrak{gl}_{m\vert n}$ and $\mathfrak{gl}_{n\vert m}$ and then construct a commutative diagram between these algebras.

    \begin{proposition}
    \label{p6}
     There exists an isomorphism between the Lie superalgebras $\mathfrak{gl}_{m\vert n}$ and $\mathfrak{gl}_{n\vert m}$.
    \end{proposition}
    \begin{proof}
    Let $X = \bordermatrix{ & m & n \cr
      m & A & B \cr
      n & C & D \cr}$, $Y = \bordermatrix{ & n & m \cr
      n & D & C \cr
      m & B & A \cr}$, $Q = \bordermatrix{ & m & n \cr
      n & 0 & I_n \cr
      m & I_m & 0 \cr}$ and $Q^{-1} = \bordermatrix{ & n & m \cr
      m & 0 & I_m \cr
      n & I_n & 0 \cr}$. Let $\Phi_Q : \mathfrak{gl}_{m\vert n} \to \mathfrak{gl}_{n\vert m}$, such that  $\Phi_Q(W) =  Q W Q^{-1}$. 
       Note that the matrix multiplication $\Phi_{Q}(W) = Q W Q^{-1}$ is well-defined, as the block partitions of $Q$, $W$, and $Q^{-1}$ are conformal by construction.
    
    Then
    \begin{align*}
    \Phi_Q (X) &= \bordermatrix{ & m & n \cr n & 0 & I_n \cr m & I_m & 0 \cr} \bordermatrix{ & m & n \cr m & A & B \cr n & C & D \cr} \bordermatrix{ & n & m \cr m & 0 & I_m \cr n & I_n & 0 \cr} \\
    &= \bordermatrix{ & m & n \cr n & C & D \cr m & A & B \cr} \bordermatrix{ & n & m \cr m & 0 & I_m \cr n & I_n & 0 \cr} \\
    &= \bordermatrix{ & n & m \cr n & D & C \cr m & B & A \cr} = Y.
    \end{align*}
    
    Note that $\Phi_Q$ preserves parity (i.e., $\Phi_Q(\mathfrak{g}_i) \subseteq \mathfrak{h}_i$ for $i=-1,0,1$), $[\Phi_Q(X), \Phi_Q(Y)] = \Phi_Q[X,Y]$ and $\Phi_Q$ is a bijection. Therefore, $\Phi_Q$ is an isomorphism.
    \end{proof}
    \begin{definition}
    Consider $W = \bordermatrix{ & m & n \cr
      m & A & B \cr
      n & C & D \cr}$. The supertranspose of $W$ denoted by $W^{st}$ has the following form:
    \[ W^{st} = \bordermatrix{ & m & n \cr
      m & A^t & C^t \cr
      n & -B^t & D^t \cr}. \]
    
    \end{definition}
    The following lemma establishes key identities for the supertranspose that will be used in the proof of the next theorem.
    \begin{lemma}
    \label{l4}
    Let $\pi$ be the projection to $\mathfrak{g}_{-1}$, $\mathfrak{t}$ be the supertranspose transformation such that $\mathfrak{t}(X) = X^{st}$ and $\Phi_Q$ be as in Proposition~\ref{p6}. \\
    Consider $X = \bordermatrix{ & m & n \cr
      m & A & B \cr
      n & C & D \cr}$ and $Y = \bordermatrix{ & n & m \cr
      n & D & C \cr
      m & B & A \cr}$. Then \[ \mathfrak{t} \circ \Phi_Q (X) = \bordermatrix{ & n & m \cr
      n & D^t & B^t \cr
      m & -C^t & A^t \cr}, \quad \pi \circ \mathfrak{t}(Y) = \bordermatrix{ & n & m \cr
      n & 0 & 0 \cr
      m & -C^t & 0 \cr}. \]
    
    \end{lemma}
    
    \begin{proof}
    From Proposition~\ref{p6}, we have $\Phi_Q(X) = Y$. Applying the supertranspose $\mathfrak{t}$ to $Y \in \mathfrak{gl}_{n\vert m}$, we transpose the diagonal blocks and supertranspose the off-diagonal blocks according to the definition, which yields:
    \[ \mathfrak{t} \circ \Phi_Q(X) = Y^{st} = \bordermatrix{ & n & m \cr
      n & D^t & B^t \cr
      m & -C^t & A^t \cr}. \]
    This establishes the first identity. For the second identity, the natural projection $\pi$ onto the $-1$ graded component isolates the lower-left block. Applying $\pi$ to $\mathfrak{t}(Y)$ therefore gives:
    \[ \pi \circ \mathfrak{t}(Y) = \bordermatrix{ & n & m \cr
      n & 0 & 0 \cr
      m & -C^t & 0 \cr}. \]
    \end{proof}

    The next lemma exhibits the formula of the dimension of some structures of our algebras.
    \begin{lemma}
      \label{l5}
Consider a family of Lie algebras $\mathfrak{g}^{B_{r}}_{-1}$ of size $m \times n$ and let $\rank B_r = r \neq 0$.
    Then\\
    $a$) $\dim \mathfrak{s}({\mathfrak{g}^{B_{r}}_{-1}}) = \dim \mathfrak{sl}(r) = r^2 -1$; \\
    $b$) $\dim \mathfrak{r}({\mathfrak{g}^{B_{r}}_{-1}}) = \dim {{\mathfrak{g}^{B_{r}}_{-1}}} - \dim \mathfrak{s}({\mathfrak{g}^{B_{r}}_{-1}}) = mn - r^2 +1$;\\
    $c$) $\dim Z(\mathfrak{g}^{B_{r}}_{-1}) = 
    \begin{cases} 
    1, & \text{if } r=m=n \\ 
    (m-r)(n-r), & \text{otherwise.} 
    \end{cases}$
    \end{lemma}
    
    \begin{proof}

    From Corollary~\ref{c2} the radical and the Levi subalgebra of these algebras have the following form:
    \begin{align*}
\mathfrak{r}\!\left(\mathfrak{g}^{B_r}_{-1}\right)
&= \left\{
\bordermatrix{ & r & m-r \cr r & \alpha I_r & y_{12} \cr n-r & y_{21} & y_{22} \cr}
\;\middle|\;
\alpha \in \mathbb{K},\; y_{12}, y_{21}, y_{22} \text{ arbitrary}
\right\}, \\[6pt]
\mathfrak{s}\!\left(\mathfrak{g}^{B_r}_{-1}\right)
&= \left\{
\bordermatrix{ & r & m-r \cr r & x_{11} & 0 \cr n-r & 0 & 0 \cr}
\;\middle|\;
x_{11} \in M_r(\mathbb{K}),\; \operatorname{tr}(x_{11}) = 0
\right\}.
\end{align*}
    
    From Corollary~\ref{c1} we have two possibilities: if $r=m=n$, the center of these algebras consists of the scalar matrices (which have dimension 1); otherwise we have that:
    \[ Z(\mathfrak{g}^{B_{r}}_{-1}) = \left\{ \bordermatrix{ & r & m-r \cr
       r & 0 & 0 \cr
      n-r & 0 & y_{22} \cr} \;\middle|\; y_{22} \in M_{(n-r)\times(m-r)}(\mathbb{K}) \right\}. \]
    Examining the structure above, the result follows from a direct calculation.\\
    \end{proof}
    
    \begin{theorem}
    \label{t7}
    Consider the Lie algebra $\mathfrak{g}^{B}_{-1}$ obtained from $\mathfrak{gl}_{m\vert n}$ and the Lie algebra $\mathfrak{h}^{H}_{-1}$ obtained from $\mathfrak{gl}_{n\vert m}$. If $\rank B = \rank H$, then
    $$ \mathfrak{g}^{B}_{-1} \simeq \mathfrak{h}^{H}_{-1}.$$
    \end{theorem}
    \begin{proof}
    By Theorem~\ref{t4}, we may assume without loss of generality that $B$ and $H$ are in normal form, which implies $H = B^t$. 
    Let $\pi$ be the canonical projection, $\mathfrak{t}$ be the supertranspose transformation such that $\mathfrak{t}(W) = W^{st}$, and $\Phi_Q$ be as in Proposition~\ref{p6}. Then we can construct the following commutative diagram:
    \[
    \begin{tikzcd}[column sep=large, row sep=large]
        \mathfrak{gl}_{m\vert n} \arrow[r, "\Phi_Q"] \arrow[d, "\pi"'] & \mathfrak{gl}_{n\vert m} \arrow[d, "\pi \circ \mathfrak{t}"] \\
        \mathfrak{g}_{-1} \arrow[r, "\mathfrak{t} \circ \Phi_Q"] & \mathfrak{h}_{-1}
    \end{tikzcd}
    \]

    Let us prove that $\mathfrak{t} \circ \Phi_Q$ is a homomorphism of Lie algebras, i.e., $\mathfrak{t} \circ \Phi_Q (\llbracket X, Y \rrbracket_B) = \llbracket \mathfrak{t} \circ \Phi_Q (X), \mathfrak{t} \circ \Phi_Q (Y) \rrbracket_{H}$.\\
    Let $X,Y \in \mathfrak{g}_{-1}$.
    Using Lemma~\ref{l3} to expand the bracket we have
    \begin{align*}
    \mathfrak{t} \circ \Phi_Q (\llbracket X, Y \rrbracket_B) &= \mathfrak{t}(Q (\llbracket X, Y \rrbracket_B) Q^{-1}) \\
    &= \mathfrak{t}\left( \bordermatrix{ & m & n \cr n & 0 & I_n \cr m & I_m & 0 \cr} \bordermatrix{ & m & n \cr m & 0 & 0 \cr n & XBY-YBX & 0 \cr} \bordermatrix{ & n & m \cr m & 0 & I_m \cr n & I_n & 0 \cr} \right) \\
    &= \mathfrak{t}\left(\bordermatrix{ & m & n \cr n & XBY-YBX & 0 \cr m & 0 & 0 \cr} \bordermatrix{ & n & m \cr m & 0 & I_m \cr n & I_n & 0 \cr}\right) \\
    &= \mathfrak{t}\left( \bordermatrix{ & n & m \cr n & 0 & XBY-YBX \cr m & 0 & 0 \cr} \right) \\
    &= \bordermatrix{ & n & m \cr n & 0 & 0 \cr m & X^tB^tY^t- Y^tB^tX^t & 0 \cr} \\
    &= \llbracket \mathfrak{t} \circ \Phi_Q (X),\mathfrak{t} \circ \Phi_Q (Y) \rrbracket_H.
    \end{align*}
    Note that the last equality holds precisely because $H = B^t$.
    Therefore, $\mathfrak{t} \circ \Phi_Q$ is a homomorphism and an isomorphism of Lie algebras.

    \end{proof}

    \begin{theorem}
    \label{t8}
    Consider the Lie algebras $\mathfrak{g}^{B}_{-1}$ obtained from $\mathfrak{gl}_{m\vert n}$   and $ \mathfrak{h}^{H}_{-1}$ obtained from $\mathfrak{gl}_{p\vert q}$, where $B$ and $H$ are as defined on Lemma~\ref{l2}, $\rank B = r_1 \ge 1$ and $\rank H = r_2 \ge 1$.\\
    Then  $$\mathfrak{g}^{B}_{-1} \simeq \mathfrak{h}^{H}_{-1} \Leftrightarrow r_1 = r_2 \text{ and } \{m,n\} = \{p,q\}.$$
    \end{theorem}
    
    \begin{proof}
    First, suppose that $\mathfrak{g}^{B}_{-1} \simeq \mathfrak{h}^{H}_{-1}$. An isomorphism of Lie algebras preserves the Levi-Malcev decomposition, meaning the semisimple Levi factors, solvable radicals, and centers must have equal dimensions. 
    
    From the dimension of the Levi factor ($\dim \mathfrak{s}(\mathfrak{g}^{B}_{-1}) = r_1^2 - 1 = r_2^2 - 1$), we immediately obtain $r_1 = r_2$. Let us denote this common rank by $r$. Equating the dimensions of the radicals ($\dim \mathfrak{r}(\mathfrak{g}^{B}_{-1}) = mn - r^2 + 1$), we get $mn = pq$. If $r \neq 0$, equating the dimensions of the centers $\dim Z = (m-r)(n-r) = mn - r(m+n) + r^2$ directly yields $m+n = p+q$, since $mn = pq$. (When $r=m=n$, we have $pq = r^2$ and $p, q \geq r \implies p=q=r$, which trivially gives $m+n=p+q$). The sum and product equalities $m+n = p+q$ and $mn = pq$ imply that the elements of the sets $\{m, n\}$ and $\{p, q\}$ are the roots of the same quadratic equation $x^2 - (m+n)x + mn = 0$. Therefore, the unordered pairs are equal: $\{m, n\} = \{p, q\}$.
    
    For the converse, if $\{m,n\} = \{p,q\}$, then either $(m,n) = (p,q)$, which yields isomorphic algebras by Theorem~\ref{t5}, or $(m,n) = (q,p)$, which yields isomorphic algebras by Theorem~\ref{t7}.
    \end{proof}
    
    \subsection{Illustrative Examples}

    \begin{example}
    Let $X = \mathfrak{g}^{B_{r_i}}_{-1}$. Then\\
    1) If $\rank B_{r_i} = 0$ then $B_{r_i} = 0$ and, from Lemma~\ref{l3}, $X$ is commutative;\\
    2) If $\rank B_{r_i} = 1$, $\dim \mathfrak{s} (X) = {r_i}^2 - 1 = 0$. Then by the Levi-Malcev decomposition we have that $X$ is solvable;\\
    3) If $m=n$ and $\rank B_{r_i}=n$ then $\dim \mathfrak{r} (X) = 1$. Note that in this case $\mathfrak{r} (X)$ consists of the scalar matrices $\mathbb{K}I_n$ and $\mathfrak{s} (X)$ is isomorphic to $\mathfrak{sl}(n)$. Since both $\mathfrak{r}(X) = \mathbb{K}I_n$ and $\mathfrak{s}(X)$ are ideals, the adjoint action of $\mathfrak{s}$ on $\mathfrak{r}$ is zero, so the semidirect sum of the Levi-Malcev decomposition becomes a direct sum and we can conclude that $X = \mathbb{K}I_n \oplus \mathfrak{sl}(n) \simeq \mathfrak{gl}(n)$.
    \end{example}
    
    \begin{example} 
    Let $X = \mathfrak{g}^{B_{2}}_{-1}$ be obtained from $\mathfrak{gl}_{5\vert 6}$ ($m=5$, $n=6$, $r=2$). 
    Using the formulas from our Levi-Malcev decomposition, we have $\dim \mathfrak{s}(X) = 3$, $\dim \mathfrak{r}(X) = 27$, and $\dim Z(X) = 12$. 
    By Theorem~\ref{t8}, any isomorphic algebra $W = \mathfrak{h}^{H}_{-1}$ (from $\mathfrak{gl}_{p\vert q}$ with $\rank H = r'$) must share these dimensions. This yields $r' = 2$, $pq = 30$, and $p+q = 11$. The unordered pair $\{p,q\}$ is uniquely determined as $\{5,6\}$ (the roots of $x^2 - 11x + 30 = 0$). Thus, the only algebra isomorphic to $X$ aside from itself is the one obtained from $\mathfrak{gl}_{6\vert 5}$ with $r=2$.
    \end{example}
    
    \begin{example}
    Let $X = \mathfrak{g}^{B_{5}}_{-1}$ be obtained from $\mathfrak{gl}_{7\vert 7}$ ($m=7$, $n=7$, $r=5$). 
    Here, $\dim \mathfrak{s}(X) = 24$, $\dim \mathfrak{r}(X) = 25$, and $\dim Z(X) = 4$. 
    Any isomorphic algebra $W$ must satisfy $r' = 5$, $pq = 49$, and $p+q = 14$. The only solution is $p=q=7$ (a double root of $x^2 - 14x + 49 = 0$). Hence, $X$ is uniquely identified up to isomorphism.
    \end{example}

    \section{Conclusion}\label{sec13}
    In this paper, we achieved a complete classification of the Lie algebras $\mathfrak{g}_{-1}^{B}$ obtained from the general linear Lie superalgebra $\mathfrak{gl}_{m\vert n}$ via the derived bracket construction generated by an odd element $B$ with $B^2 = 0$. By explicitly computing the Levi-Malcev decomposition, we demonstrated that the algebraic structure of these Lie algebras is strictly determined by the rank $r$ of the matrix representation of $B$. Specifically, we identified the semisimple Levi factor as isomorphic to $\mathfrak{sl}(r)$ and provided exact formulations for both the solvable radical and the center across arbitrary dimensions.
    
    Beyond providing a concrete structural classification, this work highlights the efficacy of the derived bracket as a tool for generating highly structured, computable Lie algebras from superalgebraic origins. The explicit matrix analysis and decomposition techniques developed here lay a solid foundation for several natural extensions.
    
    A primary direction for future research is the generalization of these structural results to fields of prime characteristic. In this modular setting, the classical Levi-Malcev decomposition theorem does not hold in general, and the structural analysis presents significant new challenges. For instance, when the characteristic $p$ of the field divides the rank $r$, the trace operator behaves degenerately; this alters the simplicity of the $\mathfrak{sl}(r)$ component (often requiring the passage to $\mathfrak{psl}(r)$) and fundamentally modifies both the semisimple Levi factor and the center of the algebra. Investigating how the derived bracket interacts with these phenomena will be a natural continuation of this work.

    \noindent\textbf{Acknowledgements.} This paper was supported by Programa
    Institucional de Auxílio à Pesquisa de Docentes Recém-Contratados ou Recém-Doutorados, Universidade Federal de Minas Gerais 2018. I would like to thank my supervisor E. Vishnyakova for the efforts towards my formation. Without her constant support and guidance this paper would not have been possible. Many thanks to C. Schneider for the discussion about the case when the characteristic of $\mathbb{K}$ is prime and to R. Biezuner for pointing errors in the calculations and suggesting improvements.

    \end{document}